\documentclass{amsart}

\usepackage{}
\usepackage{trajan}
\usepackage{multirow}
\usepackage{latexsym,array,delarray,amsthm,amssymb,epsfig}
\usepackage{amsmath}

\theoremstyle{plain}
\newtheorem{thm}{Theorem}[section]
\newtheorem{lemma}[thm]{Lemma}
\newtheorem{prop}[thm]{Proposition}
\newtheorem{cor}[thm]{Corollary}

\newtheorem*{thm*}{Theorem}
\newtheorem*{lemma*}{Lemma}
\newtheorem*{prop*}{Proposition}
\newtheorem*{cor*}{Corollary}
\newtheorem*{conj*}{Conjecture}

\theoremstyle{definition}
\newtheorem{defn}[thm]{Definition}

\theoremstyle{remark}

\numberwithin{equation}{section}

\begin{document}

\date{}

\title{ON NILPOTENT LIE ALGEBRAS OF SMALL BREADTH}

\author{Borworn Khuhirun, Kailash C. Misra and Ernie Stitzinger}
\address{Department of Mathematics, North Carolina State University, Raleigh, NC 27695-8205}
\email{duomaxwelldv@hotmail.com, misra@ncsu.edu, stitz@ncsu.edu}
\thanks{KCM is partially supported by NSA grant \#  H98230-12-1-0248 and Simons Foundation grant \#  307555}
\begin{abstract}
A Lie algebra $L$ is said to be of breadth $k$ if the maximal dimension of the images of left multiplication by elements of the algebra is $k$. In this paper we give characterization of finite dimensional nilpotent Lie algebras of breadth less than or equal to two. Furthermore, using these characterizations we determined the isomorphism classes of these algebras.
\end{abstract}

\maketitle
\bigskip
\section{Introduction}
Classification of algebraic objects is a central theme of mathematical research. The classification of finite dimensional complex simple Lie algebras due to Killing and Cartan is well known (cf. \cite{H}). However, due to the existence of the vast number of finite dimensional nilpotent Lie algebras the classification problem has been formidable for this class. As a result, many researchers have made progress by classifying nilpotent Lie algebras satisfying certain conditions. Research in finite group theory has followed a similar path. Simple groups have been classified, but the large number of finite $p$-groups has led researchers to investigate $p$-groups with added conditions. One such condition is the breadth of a finite $p$-group (\cite{N},\cite{V}). The breadth $b(G)$ of a finite $p$-group $G$ is defined as  the size of the largest conjugacy class in $G$. Analogously the breadth $b(L)$ of a Lie algebra $L$ is defined to be the maximum of the dimensions of the images of ${\rm ad}_x$ for all $x \in L$. In \cite{V}, Vaughn-Lee settled a long standing conjecture by showing that for a finite $p$-group $G$ of breadth $b=b(G)$,  $|G^2| \leq p^(b(b+1)/2)$. He also showed that for a  nilpotent Lie algebra $L$ of breadth $b=b(L)$, ${\rm dim}[L, L] \leq  b(b+1)/2$. More recently, Parmeggiani and Stellmacher \cite{PS}  gave characterizations of finite $p$-groups of breadth $1$ and $2$. However, so far there does not  exist a classification of these finite $p$-groups. 

In this paper we give  characterizations for nilpotent Lie algebras of breadth $1$ and $2$.  In particular, we show that a finite dimensional nilpotent Lie algebra is of breadth $1$ if and only if its derived algebra is one dimensional.  We also show that a finite dimensional nilpotent Lie algebra $L$  has breadth $2$ if and only if either the derived algebra of $L$ has dimension $2$ or the derived algebra and the central quotient both have dimension $3$. These results parallel results in finite $p$-groups.

Finally we use our characterizations to classify finite dimensional nilpotent Lie algebras of breadth $1$ and $2$. We define a Lie algebra to be {\it pure} if it does not have abelian ideals as direct summands. Then we classify finite dimensional pure nilpotent Lie algebras of breadth one and two since abelian summands can be added harmlessly. In particular, we show that a finite dimensional pure nilpotent Lie algebra of breadth $1$ is isomorphic to a Heisenberg Lie algebra. For a finite dimensional pure nilpotent Lie algebras $L$, the center is contained in the derived algebra.  By our characterization result, the dimension of  the derived algebra of a finite dimensional pure nilpotent Lie algebra $L$ of breadth $2$ is either $2$ or $3$. We determine the isomorphism classes of finite dimensional pure nilpotent Lie algebras of breadth two with three dimensional derived algebra. We also determine the isomorphism classes of finite dimensional pure nilpotent Lie algebras of breadth two with two dimensional derived algebra and one dimensional center. For the remaining case where the derived algebra and center have dimension $2$ each, we determine their isomorphism classes up to dimension 6. We hope that these classification results will lead to corresponding classification results in finite $p$-groups.

\section{Basic properties of breadth for Lie algebras}
Let $L$ be a finite dimensional Lie algebra over field $F$ and $A$ be an ideal of $L$. In this paper we assume ${\rm char}(F) \neq 2$. For any $x \in L$ we define $b_A(x) = {\rm rank}({\rm ad}_x|_A)$ and $b_A(L) = {\max}\{b_A(x) \mid x \in L\}$. Clearly $b_A(L) \leq {\rm dim}[A,L]$. We define the ${\it breadth}$ of $L$ to be $b(L) = b_L(L)$. We denote $b(x) = b_L(x)$ for all $x \in L$. Clearly $b_A(L) \leq b(L)$. It follows from the definition that $L$ is abelian if and only if $b(L) =0$. Let $Z(L)$ denote the center of $L$. The following results are easy to prove.
\begin{prop}\label{lowerbound}
${\rm dim}(L/Z(L)) \geq b(L)+1$. 
\end {prop}
\begin{prop}\label{directsum} Let $L = L_1 \oplus L_2$ be a finite dimensional direct sum Lie algebra. Then $b(L) = b(L_1) + b(L_2)$.
\end {prop}
We also have the following known result.
\begin{prop}\cite{V}\label{upperbound} Let $L$ be a finite dimensional nilpotent Lie algebra. Then ${\rm dim}[L, L] \leq  b(L)(b(L)+1)/2$.
\end {prop}

Now we have the following characterization of breadth one finite dimensional Lie algebras.
\begin{thm}\label{one-char} 
$b(L) =1$ if and only if ${\rm dim}[L,L] = 1$.
\end{thm}
\begin{proof} Suppose ${\rm dim}[L,L] = 1$. Since $b(L) \leq {\rm dim}[L,L] = 1$ and $L$ is not abelian we have $b(L) = 1$. Conversely, suppose $b(L) =1$ and ${\rm dim}[L,L] \geq 2$. Choose $x_1, y_1, x_2, y_2, z_1, z_2 \in L$ such that $[x_1,y_1] = z_1, [x_2,y_2] =z_2$ and $z_1, z_2$ linearly independent. Since $b(L) = 1$, we have $b(x_1) = 1 = b(x_2)$ which implies that $[x_1, y_2] = 0 = [x_2, y_1]$. Hence $[x_1+x_2, y_1] = z_1, [x_1+x_2, y_2] = z_2$ which implies $b(x_1+x_2) \geq 2$, a contradiction. Since ${\rm dim}[L,L] \geq 1$, it follows that ${\rm dim}[L,L] = 1$.
\end{proof}
Now we give the following classification of finite dimensional nilpotent Lie algebras with breadth one.
\begin{thm}\label{one-class}
Let $L$ be a $n$-dimensional nilpotent Lie algebra with $b(L)=1$. Then $L$ has a basis $\{x_1, y_1, x_2, y_2, \cdots , x_k, y_k, z_1, z_2, \cdots , z_{n-2k}\}$ with $[x_i, y_j] = \delta_{ij} z_1$ and $[z_i, L] = 0$. 
\end{thm}
\begin{proof}
Since $b(L) = 1$, by Theorem \ref{one-char}  ${\rm dim}[L,L] = 1$. Since $L$ is nilpotent $[L, L] = Z(L)$. Let $[L, L] = {\rm span}\{z_1\}$ and $V$ be its complementary subspace in $L$. For $u, v \in V$, we have $[u, v] = \alpha z_1$ for some $\alpha \in F$. We define the bilinear form $\varphi : V \times V \longrightarrow F$ by 
$\varphi (u, v) = \alpha$. This is an alternating form since by definition $\varphi (u, u) = 0$ for all $u \in V$. Since ${\rm char}(F) \neq 2$, by (cf. \cite{J}, Theorem 6.3) there exists a basis $\{x_1, y_1, x_2, y_2, \cdots , x_k, y_k, z_2, \cdots , z_{n-2k}\}$ for $V$ such that the matrix of $\varphi$ with respect to this basis is ${\rm diag}\{S_1,S_2,\ldots,S_k,0,\ldots,0\}$ where $S_i=
\begin{pmatrix}
0&1\\
-1&0
\end{pmatrix}$ 
for $i = 1, 2, \cdots , k$.
Therefore, $L$ has the basis $\{x_1, y_1, x_2, y_2, \cdots , x_k, y_k, z_1, z_2, \cdots , z_{n-2k}\}$ with $[x_i, y_j] = \delta_{ij} z_1$ and $[z_i, L] = 0$ by definition of the bilinear form $\varphi$. 
\end{proof}
\section{Characterizations of breadth two nilpotent Lie algebras}
For the rest of this paper we focus on finite dimensional nilpotent Lie algebras $L$ over $F$ with $b(L) = 2$. Then $\{0\} \subsetneq Z(L) \subsetneq L$. In this section we prove the following Theorem which gives a characterization of such Lie algebras.
\begin{thm}\label{char}
Let $L$ be a finite dimensional nilpotent Lie algebra. Then $b(L) = 2$ if and only if one of the following conditions hold:
\begin{enumerate}
\item ${\rm dim}[L, L] =2$, or
\item ${\rm dim}[L, L] =3$ and ${\rm dim}(L/Z(L)) =3$.
\end{enumerate}
\end{thm}
If $L$ is a finite dimensional nilpotent Lie algebra with ${\rm dim}[L, L] =2$, then $b(L) \neq 0$ since $L$ is not abelian and $b(L) \neq 1$ 
by Theorem \ref{one-char}. Hence $b(L) = 2$ since $b(L) \leq {\rm dim}[L, L]$. Now suppose ${\rm dim}[L, L] =3$ and ${\rm dim}(L/Z(L)) =3$. Then 
$b(L) \leq {\rm dim}[L, L] = 3$ and as before $b(L) \neq 0, 1$. If $b(L) = 3$, then by Proposition \ref{lowerbound} ${\rm dim}(L/Z(L)) \geq 4$ which is a contradiction. Hence $b(L) =2$ proving one direction of Theorem \ref{char}. 
In order to prove the other direction of this theorem we need the following results.

For the rest of this section we assume $L$ is a finite dimensional nilpotent Lie algebra with $b(L) =2$ and $A$ denote a maximal abelian ideal of $L$. Then $Z(L) \subseteq A \subsetneq L$ and $C_L(A) = A$. Let $T_A = \{x \in L \mid b_A(x) =1\}$.

\begin{lemma}\label{six elements}
Let $L$ be a finite dimensional nilpotent Lie algebra with $b(L) =2$, $A$ be a maximal abelian ideal of $L$ and $b_A(L) =2$. Let $x, y, z \in L$ such that $y - z \notin A$, and $b_A(x) = 2$. Then at least one of the elements $ y, z, y+z, x+y, x+z, x+y+z$ is not in $T_A$. 
\end{lemma}
\begin{proof} Suppose $ S = \{y, z, y+z, x+y, x+z, x+y+z\} \subset T_A$. Set $T_1 = {\rm span}\{y, z\} = {\rm span}\{y, y-z\}, T_2 = {\rm span}\{x+y, x+z\} = {\rm span}\{y-z, x+z\}$ and $T_3 = {\rm span}\{y, x+z\}$. Note that $b_A(y-z) = 1$ since $y-z \notin A$. Since $b_A(L) = 2$, we have ${\rm dim}[T_j, A] = 1, \, {\rm or} \, 2$ for $1 \leq j \leq 3$. So at least two of these subspaces have dimension $1$ or at least two of these subspaces have dimension $2$.

Suppose ${\rm dim}[T_1, A] = 1$ and ${\rm dim}[T_2, A] = 1$.  Thus $[A, y] = [A, z] = [A, y-z] = [A, x+y] = [A, x+z]$ and each of them is one dimensional. Then $[A, x] = [A, (x+z) + (y-z) - y] \subseteq [A, x+z] + [A, y-z] + [A, -y] \subseteq [A, y]$ contradicting the assumption that $b_A(x) = 2$. Similar augments show that if ${\rm dim}[T_1, A] = 1 = {\rm dim}[T_3, A]$ or ${\rm dim}[T_2, A] = 1 = {\rm dim}[T_3, A]$ we have a contradiction.

Suppose ${\rm dim}[T_1, A] = 2 = {\rm dim}[T_2, A]$. Since $b_A(y) = b_A(z) = 1$ and ${\rm dim}[T_1, A] = 2$, we have $\{y, z\}$ is linearly independent and ${\rm im}({\rm ad}_y|_A) \cap {\rm im}({\rm ad}_z|_A) = \{0\}$. Let $a \in {\rm ker}({\rm ad}_{y-z}|_A)$. Then $[a, y-z] = 0$ which implies $[a, y] = [a, z] = 0$ since ${\rm im}({\rm ad}_y|_A) \cap {\rm im}({\rm ad}_z|_A) = \{0\}$. Therefore, $a \in {\rm ker}({\rm ad}_{y}|_A)$ and 
$a \in {\rm ker}({\rm ad}_{z}|_A)$ Hence ${\rm ker}({\rm ad}_{y-z}|_A) \subseteq {\rm ker}({\rm ad}_{y}|_A)$ and ${\rm ker}({\rm ad}_{y-z}|_A) \subseteq {\rm ker}({\rm ad}_{z}|_A)$. This implies ${\rm ker}({\rm ad}_{y-z}|_A) = {\rm ker}({\rm ad}_{y}|_A) = {\rm ker}({\rm ad}_{z}|_A)$ since each of them have codimension $1$ in $A$. Similarly, since ${\rm dim}[T_2, A] = 2, b_A(x+y) = 1 = b_A(x+z)$ we have $\{x+y, x+z\}$ linearly independent and ${\rm im}({\rm ad}_{x+y}|_A) \cap {\rm im}({\rm ad}_{x+z}|_A) = \{0\}$. Since $y-z = (x+y) - (x+z)$, as before we have
${\rm ker}({\rm ad}_{y-z}|_A) = {\rm ker}({\rm ad}_{x+y}|_A) = {\rm ker}({\rm ad}_{x+z}|_A)$. In particular, ${\rm ker}({\rm ad}_{y}|_A) = {\rm ker}({\rm ad}_{x+y}|_A)$. Since $x = (x+y) - y$ it follows that ${\rm ker}({\rm ad}_{y}|_A) = {\rm ker}({\rm ad}_{x+y}|_A) \subseteq {\rm ker}({\rm ad}_{x}|_A)$ which implies ${\rm rank}({\rm ad}_x|_A) \leq 1$, a contradiction since $b_A(x) = 2$ by assumption. Similar augments show that  ${\rm dim}[T_1, A] = 2 = {\rm dim}[T_3, A]$ or ${\rm dim}[T_2, A] = 2 = {\rm dim}[T_3, A]$ leads to a contradiction. Therefore, $S \not\subseteq T_A$ which proves the result.
\end{proof}

\begin{lemma}\label{dimension}
Let $L$ be a finite dimensional nilpotent Lie algebra with $b(L) =2$ and $A$ be a maximal abelian ideal of $L$. Suppose $b_A(L) = 2$. 
Then ${\rm dim} [A, L] = 2$.
\end{lemma}
\begin{proof}
Since $b_A(L) = 2$, there exists $x \in L \setminus A$ such that $b_A(x) =2$. Define $L_x = \cap_{c \in A}(A + {\rm ker}({\rm ad}_{c+x}))$. Clearly 
$A \subseteq L_x \subseteq L$. For any $c \in A$, $b_A(x) = b_A(c+x) = 2 = b(L)$ since $A$ is abelian. Hence ${\rm im}({\rm ad}_{c+x}) = {\rm im}({\rm ad}_{c+x})|_A$ for all $c \in A$. Therefore, for any $y \in L$ there exists $b \in A$ such that $[c+x, y] = [c+x, b]$ which implies $y - b \in {\rm ker}({\rm ad}_{c+x})$ and $y = b + (y - b) \in A+ {\rm ker}({\rm ad}_{c+x})$ for all $c \in A$. Hence $y \in L_x$ and we have $L = L_x$.

Choose any elements $a \in A$ and $y \in L = L_x$. So $y = a'+y' = a''+y''$ where $a', a'' \in A$ and $ y' \in {\rm ker}({\rm ad}_{a+x}), y'' \in {\rm ker}({\rm ad}_{x})$.  Hence $y' - y'' \in A, [a+x, y'] = 0 = [x, y'']$. So we have $[a, y] = [a, a'+y'] = [a, y'] = [y', x] = [y'-y'', x] \in [A, x]$. Therefore, $[A, L] \subseteq [A, x]$ which implies that ${\rm dim}[A, L] \leq b_A(x) = 2$. However, since $b_A(L) = 2$ we have ${\rm dim}[A, L] \geq 2$. Therefore, ${\rm dim}[A, L] = 2$.
\end{proof}
\begin{prop}\label{derived2}
Let $L$ be a finite dimensional nilpotent Lie algebra with $b(L) =2$ and $A$ be a maximal abelian ideal of $L$. Suppose $b_A(L) =2$ and $[x,  L] \subseteq [A, L]$ for all $x \in L$ with $b_A(x) = 2$. Then $[L, L] = [A, L]$ and ${\rm dim}[L, L] = 2$.
\end{prop}
\begin{proof} By Lemma \ref{dimension} it suffices to prove that  $[L, L] = [A, L]$. Since $[A, L] \subseteq [L, L]$, we need to show that $[L, L] \subseteq [A, L]$. Suppose $[y, L] \not\subseteq [A, L]$ for some $y \in L \setminus A$. Then there exists $z \in L \setminus A$ such that $[y, z] \notin [A, L]$. If $b_A(y) = 0$, then $y \in C_L(A) =A$ which is a contradiction. If $b_A(y) = 2$, then by assumption $[y, L] \subseteq [A, L]$ which is a contradiction. Hence $b_A(y) =1$. Similarly $b_A(z) \neq 0$ since $z \notin A$. If $b_A(z) = 2$, then by assumption $[z, L] \subseteq [A, L]$. Hence $[y, z] = -[z, y] \in [A, L]$ which is a contradiction. Therefore, $b_A(z) =1$. Also note that $y - z \notin A$ since $y - z \in A$ implies $[y, z] = -[y, y - z] \in [L, A] = [A, L]$ which is a contradiction.

Since $b_A(L) = 2$, there exists $x \in L$ such that $b_A(x) = 2$. Since $A$ is abelian $x \notin A$. By assumption $[x,  L] \subseteq [A, L]$. If $b_A(x+y) = 0$, then $x+y \in C_L(A) = A$ which implies $[y, L] = [-x, L] + [x+y, L] \in [A, L]$, a contradiction. If $b_A(x+y) = 2$, then by assumption $[x+y, L] \subseteq [A, L]$ which implies $[y, L] = [-x, L] + [x+y, L] \in [A, L]$, a contradiction. Therefore, $b_A(x+y) = 1$ since $b_A(x+y) \leq 1$. By similar argument it follows that $b_A(x+z) = 1$. Now suppose $b_A(y+z) = 0$. Then as before $y+z \in C_L(A) = A$ which implies $[y, z] = [y, y+z] \in [L, A] = [A, L]$, a contradiction. If $b_A(y+z) = 2$, then by assumption $[y+z, L] \subseteq [A, L]$ which implyes $[y, z] = [y+z, z] \in [A, L]$, a contradiction. Therefore, $b_A(y+z) = 1$ since $b_A(y+z) \leq 1$. Now suppose $b_A(x+y+z) = 0$. Then as before $x+y+z \in C_L(A) = A$ which implies $[y, z] = [x+y+z, z] - [x, z] \in [A, L]$, a contradiction. If $b_A(x+y+z) = 2$, then by assumption $[x+y+z, L] \in [A, L]$ which implies $[y, z] = [x+y+z, z] - [x, z] \in [A, L]$, a contradiction. Hence $b_A(x+y+z) =1$ since $b_A(x+y+z) \leq 2$. However, by Lemma \ref{six elements} this is a contradicttion. Therefore, $[L, L] \subseteq [A, L]$, which completes the proof.
\end{proof}
\begin{prop}\label{centralquotient}
Let $L$ be a finite dimensional nilpotent Lie algebra with $b(L) =2$ and $A$ be a maximal abelian ideal of $L$. Suppose $b_A(L) = 1$. Then one of the following hold:
\begin{enumerate}
\item  ${\rm dim}[A, L] =1$ and $b(L/[A, L]) < 2$.
\item ${\rm dim}(A/Z(L)) =1$ and  ${\rm dim}(L/Z(L)) \leq 3$.
\end{enumerate}
\end{prop}
\begin{proof} 
Since $b_A(L) = 1$ and $A = C_L(A)$, for all $x \in L \setminus A$ we have $b_A(x) = 1$ and ${\rm nullity}({\rm ad}_x|_A) = {\rm dim}(A) - 1$.
Suppose that there exists elements $x, y \in L\setminus A$ such that ${\rm ker}({\rm ad}_x|_A) \neq {\rm ker}({\rm ad}_y|_A)$. Then $A = {\rm ker}({\rm ad}_x|_A) + {\rm ker}({\rm ad}_y|_A)$ and there exists elements $a, b \in A$ such that $[x, a] = 0 = [y, b]$ but $[x, b] \neq 0, [y, a] \neq 0$. Since $b_A(L) = 1$, it follows that $[x, b]$ and $[y, a]$ are linearly dependent. Hence $[x, A] = [y, A]$ since $b_A(x) = 1 = b_A(y)$. Let $z \in L \setminus A$ and $[z, A] \neq [x, A] = [y, A]$. Then by the above argument ${\rm ker}({\rm ad}_z|_A) = {\rm ker}({\rm ad}_x|_A)$ and ${\rm ker}({\rm ad}_z|_A) = {\rm ker}({\rm ad}_y|_A)$. This is a contradiction since by assumption ${\rm ker}({\rm ad}_x|_A) \neq {\rm ker}({\rm ad}_y|_A)$. Therefore, $[z, A] = [x, A]$ for all $z \in L \setminus A$. Hence ${\rm dim}[L, A] = {\rm dim}[A, L] = 1$ since $b_A(L) = 1$. Now let $ x + [A, L] \in L/ [A, L]$. Then $[x, A] \subseteq [x, L]$ and ${\rm dim}[x, A] = 1$ since $b_A(x) =1$.  Hence ${\rm dim}[x, L]/[x, A] \leq 1$ since $b(L) = 2$. This implies $b(L/[A, L]) <2$.

Now suppose that ${\rm ker}({\rm ad}_x|_A) = {\rm ker}({\rm ad}_y|_A)$ for all $x, y \in L\setminus A$. Since $Z(L) \subseteq A$, we have $Z(L) \subset {\rm ker}({\rm ad}_x|_A)$ for all $x \in L \setminus A$. Let $z \in {\rm ker}({\rm ad}_x|_A)$. Then $[z, y] =0$ for all $y \in L$ since $A$ is abelian and ${\rm ker}({\rm ad}_x|_A) = {\rm ker}({\rm ad}_y|_A)$ for all $y \in L\setminus A$. Therefore, ${\rm ker}({\rm ad}_x|_A) = Z(L)$  and 
${\rm dim}(Z(L)) = {\rm nullity}({\rm ad}_x|_A) = {\rm dim}(A) - 1$ for all $x \in L \setminus A$. Hence ${\rm dim}(A/Z(L)) = 1$. Choose $u \in A \setminus Z(L)$. Then $A = {\rm span}\{u, Z(L)\}$. Suppose ${\rm dim}(L/A) \geq 3$. Then there exist $x, y, z \in L \setminus A$ such that $\{x+A, 
y+A, z+A\}$ is linearly independent. Suppose $k_1[x, u] + k_2[y, u] + k_3[z, u] = 0$ for some $k_1, k_2, k_3 \in F$. Then $[k_1x +k_2y +k_3z, u] = 0$ which implies $[k_1x +k_2y +k_3z, A] = 0$. Hence $k_1x +k_2y +k_3z \in C_L(A) = A$ which is a contradiction. Therefore, $\{[x,u], [y, u], [z, u]\}$ is linearly independent and $b(u) \geq 3$ which is a contradiction since $b(L) = 2$. Hence   ${\rm dim}(L/A) \leq 2$ which implies  ${\rm dim}(L/Z(L)) \leq 3$.
\end{proof}

{\bf Proof of Theorem \ref{char}:} We assume that $b(L) =2$. Then by Proposition \ref{upperbound} we have ${\rm dim}[L,L] \leq 3$. Since $L$ is not abelian ${\rm dim}[L,L] \geq1$ and by Theorem \ref{one-char} ${\rm dim}[L,L] \neq 1$. Hence ${\rm dim}[L,L] = 2 \, {\rm or} \, 3$. Suppose that ${\rm dim}[L,L] = 3$. Since $b_A(L) \leq b(L) =2$, we have $b_A(L) = 0, 1 \, {\rm or} \, 2$. If $b_A(L) = 0$, then $A = Z(L)$. Since $A = C_L(A)$, it follows that $A = L$. This is a contradiction since $b(L) = 2$. If $b_A(L) =2 = b(L)$, then for all $x \in L$ with $b_A(x) =2$ we have $[L, x] = [A, x] \subseteq [A, L]$. Hence by Proposition \ref{derived2}, we have ${\rm dim}[L, L] = 2$ which is a contradiction. Finally suppose $b_A(L) = 1$. Then by Proposition \ref{centralquotient}, either ${\rm dim}(L/Z(L)) \leq 3$ or ${\rm dim}[A, L] =1$ and $b(L/[A, L]) < 2$. In the later case $b(L/[A, L]) = 0 \, {\rm or} \, 1$. If $b(L/[A, L]) = 0$, then $L/[A, L]$ is abelian which implies that for all $x \in L$ we have ${\rm Im}({\rm ad}_x) \subseteq [A, L]$. Since in this case we also have ${\rm dim}[A, L] = 1$, it follows that $b(L) = 1$ which is a contradiction. If $b(L/[A, L]) = 1$, then by 
Theorem \ref{one-char} ${\rm dim}([L, L]/[A, L]) = 1$ which implies that ${\rm dim}[L, L] = 2$ since ${\rm dim}[A, L] = 1$. This is a contradiction. Therefore,  ${\rm dim}(L/Z(L)) \leq 3$. Hence by Proposition \ref{lowerbound} we have ${\rm dim}(L/Z(L)) = 3$, proving the theorem.

\begin{cor}\label{derived3} If ${\rm dim}[L, L] = 3$, the ${\rm dim}(A/Z(L)) = 1$.
\end{cor}
\begin{proof}
As shown in the proof of Theorem \ref{char} above, if ${\rm dim}[L, L] = 3$, then $b_A(L) =1$ and ${\rm dim}(L/Z(L)) \leq 3$. Hence by Proposition \ref{centralquotient} we have ${\rm dim}(A/Z(L)) = 1$.
\end{proof}
\section{Classification of breadth two nilpotent Lie algebras}
In this section we use the characterization given in Theorem \ref{char} to determine the isomorphism classes of finite dimensional nilpotent Lie algebras $L$ with $b(L) = 2$. Note that in this case by Proposition \ref{lowerbound} we have ${\rm dim}(L/Z(L)) \geq 3$. Hence ${\rm dim}(L) \geq 4$ since $Z(L) \neq \{0\}$. If ${\rm dim}(L) = 4$, then ${\rm dim}Z(L) = 1$ and by Theorem \ref{char} we have ${\rm dim}[L, L] =2$. In the following we only state the nonzero brackets for the basis vectors of a Lie algebra.
\begin{thm}\label{dim4}
Let $L$ be a $4$-dimensional nilpotent Lie algebra with $b(L) =2$. Then $L = {\rm span}\{x_1, x_2, x_3, z\}$ with 
$[x_1, x_2] = x_3, [x_1, x_3] = z$.
\end{thm}
\begin{proof}
Since $L$ is a nilpotent Lie algebra of dimension $4$, by Theorem \ref{char} we have ${\rm dim}[L, L] = 2$. Suppose $Z(L) \not\subseteq [L, L]$. Then $L$ has an abelian ideal $I$ as a direct summand. This implies $L = L' \oplus I$ where $L'$ is nilpotent of dimension at most $3$. Hence $b(L') \leq 1$ and $b(I) = 0$. This means by Proposition \ref{directsum} that $b(L) \leq 1$ which is a contradiction. Hence $Z(L) \subseteq  [L, L]$. Suppose $Z(L) = [L, L]$. Choose a basis $\{u, v\}$ for $Z(L)$ and extend it to a basis $\{x, y, u, v\}$ of $L$. Then $[L, L] = {\rm span}\{[x, y]\}$ contradicting the fact that ${\rm dim}[L, L] = 2$. Therefore ${\rm dim}Z(L) = 1$. Hence $L/Z(L)$ is a three dimensional nilpotent Lie algebra which is not abelian since $Z(L) \subsetneq [L, L]$. Thus $L/Z(L)$ is a three dimensional Heisenberg Lie algebra. Choose a basis $\{x+Z(L), y+Z(L), v+Z(L)\}$ with $\{[x, y] - v, [x, v], [y, v]\} \subset Z(L)$. Suppose $Z(L) = {\rm span}\{u\}$. Then $[x, y] = v + au, [x, v] = bu, [y, v] = cu$ for some $a, b, c \in F$. Setting $w = v+au$, we have $[x, y] = w, [x, w] =bu, [y, w] = cu$. Note that since ${\rm dim}[L, L] = 2$, both $b$ and $c$ can not be zero. If $b = 0$ and $c = 0$, Then $[L, L] = {\rm span}\{w\}$ which is a contradiction. If $b \neq 0, c = 0$, then we set $x_1 = x, x_2 = y, x_3 = w, z = bu$. The case $b = 0, c \neq 0$ is similar. If $b \neq 0 $ and $c \neq 0$, then set $x_1 = x, x_2 = by - cx, x_3 = bw, z = b^2u$ to complete the proof.
\end{proof}
\begin{defn} A Lie algebra is said to be ${\it pure}$ if it does not have an abelian ideal as a direct summand.
\end{defn}
It follows immediately from the definition that $L$ is pure if and only if $Z(L) \subseteq [L, L]$.
Since we can harmlessly add abelian ideals as direct summands, it is sufficient to classify the finite dimensional pure nilpotent Lie algebras. In particular, it follows from Theorem \ref{one-class} that any finite dimensional pure nilpotent Lie algebra $L$ with $b(L) =1$ is isomorphic to a Heisenberg Lie algebra. 
\begin{prop}\label{reducible} Let $L$ be a finite dimensional pure nilpotent Lie algebra with $b(L) = 2$ and suppose $L = L_1 \oplus L_2 \oplus \cdots \oplus L_n$ is a direct sum of ideals. Then $n = 2$ and each $L_i$ is isomorphic to a Heisenberg Lie algebra. (Hence ${\rm dim}(L)$ is even.)
\end{prop}
\begin{proof}
By Proposition \ref{directsum} it follows that $\sum_{i=1}^nb(L_i) = b(L) = 2$. Since $L$ is pure $b(L_i) \neq 0, 1\leq i \leq n$ as $L_i$ is not an abelian ideal. Hence $n=2$ and $b(L_i) = 1, i = 1, 2$. So by Theorem \ref{one-class} , each $L_i, i = 1, 2$ is isomorphic to a Heisenberg Lie algebra. 
\end{proof}
Now using Theorem \ref{char} we will classify finite dimensional pure nilpotent Lie algebras $L$ with $b(L) = 2$. By Theorem \ref{dim4}  and the previous observation we can assume that ${\rm dim}(L) \geq 5$. By Proposition \ref{reducible} we also assume that $L$ is irreducible (i.e. not a direct summand of Lie algebras of smaller dimensions).

\begin{thm}
Let $L$ be a finite dimensional pure nilpotent Lie algebra with $b(L) =2$ and ${\rm dim}[L, L] =3 = {\rm dim}L/Z(L)$. Then
\begin{enumerate}
\item $L={\rm span}\{x_1,x_2,x_3,z_1,z_2\}$ with $[x_1,x_2]=x_3, [x_1,x_3]=z_1, [x_2,x_3]=z_2$,or
\item $L={\rm span}\{x_1,x_2,x_3,z_1,z_2,z_3\}$ with $[x_1,x_2]=z_1, [x_1,x_3]=z_2,[x_2,x_3]=z_3$.
\end{enumerate}
\end{thm}
\begin{proof}
Since $L$ is pure we have $Z(L) \subseteq [L, L]$, so ${\rm dim}Z(L) = 1, 2$ or $3$. If ${\rm dim}Z(L) = 1$, then ${\rm dim}(L) = 4$ and $[L, L]$ has codimension $1$ in $L$ which contradicts nilpotency of $L$. Suppose ${\rm dim}Z(L) = 2$. Then ${\rm dim}(L) = 5$. Choose a basis $\{u,v\}$ for $Z(L)$ and extend it to a basis $\{x, y, z, u, v\}$ of $L$ where $\{z, u, v\}$ is a basis for $[L, L]$. Since $L/Z(L)$ is a $3$-dimensional nilpotent Lie algebra we have $[L, [L, L]] \subseteq Z(L)$, hence $[x, z] =b_1u+b_2v$ and $[y, z] =c_1u+c_2v$ for some $b_1, b_2, c_1, c_2 \in F$. Also $[x, y] =a_0z+ a_1u+a_2v$ for some $a_0, a_1, a_2 \in F$ where $a_0 \neq 0$ since $[L, L] \supsetneq Z(L)$. Now setting $x_1 =x, x_2 = y, x_3 = a_0z+ a_1u+a_2v, z_1 = a_0(b_1u+b_2v), z_2 = a_0(c_1u+c_2v)$ we obtain the algebra $L$ in $(1)$. 

Now we assume ${\rm dim}Z(L) = 3$. Then ${\rm dim}(L) = 6$ and $Z(L) = [L, L]$. Choose a basis $\{u,v, w\}$ for $Z(L)$ and extend it to a basis $\{x, y, z, u, v, w\}$ of $L$. Since ${\rm dim}[L, L] = 3$, and $Z(L) = [L, L]$, we observe that $\{z_1= [x, y], z_2 = [y, z], z_3 = [z, x]\}$ is a basis for 
$Z(L)$. Now setting $x_1 = x, x_2 = y, x_3 = z$ we easily see that $L = {\rm span}\{x_1, x_2, x_3, z_1, z_2, z_3\}$ with the bracket structure given in $(2)$.
\end{proof}
\begin{thm}
Let $L$ be a finite dimensional pure nilpotent Lie algebra with ${\rm dim}(L) \\ = n+4, n \geq 1$, $b(L) =2$, ${\rm dim}[L, L] =2$ and  ${\rm dim}Z(L) =1$. Then 
$L={\rm span}\{x_1,x_2,x_3,$ $z, z_1,\cdots, z_n\}$ with the nonzero products given by:
\begin{enumerate}
\item $n$ even: $[x_1, x_2] = x_3, [x_1, x_3] = z, [z_i, z_{i+1}] = z, i = 1, 3, 5, \cdots , n-1$.
\item $n$ odd: $[x_1, x_2] = x_3, [x_1, x_3] = z, [x_2, z_1] = z, [z_{i+1}, z_{i+2}] = z, \\ i = 1, 3, 5, \cdots , n-2$.
\end{enumerate} 
\end{thm}
\begin{proof}
Since $L$ is a pure nilpotent Lie algebra, we have $Z(L) \subseteq [L, L] = L^2$ and $L^3 = [L, [L, L]] \subsetneq L^2$.  Therefore, ${\rm dim}(L^3) \leq 1$. Since $L^2 \not\subseteq Z(L)$, we have $L^3 \neq 0$. Hence $L^3 = Z(L)$ since $L^3 \cap Z(L)\neq \{0\}$. Choose a basis $\{u\}$ for $Z(L)$ and extend it to a basis $\{v, u\}$ for $L^2$. For any $y \in L$, we have ${\rm ad}_y|_{L^2} : L^2 \longrightarrow L^3 = Z(L)$. Since ${\rm dim}Z(L) =1$ and $[v, u] = 0$, the subspace ${\rm span}\{{\rm ad}_y|_{L^2} \mid y \in L\}$ of $gl(L^2)$ is one dimensional. Hence $C_L(L^2)$ which is the kernel of the map ${\rm ad}: y \longrightarrow {\rm ad}_y|_{L^2} $ has codimension $1$ in $L$. Choose $x \in L \setminus C_L(L^2)$. Then $L = {\rm span}\{x\} \oplus C_L(L^2)$. Since $ 0 \neq [x, v] \in Z(L)$, scaling $u$ if necessary we can assume that $[x, v] = u$. For $v_1, v_2 \in C_L(L^2)$, we have $[x, [v_1, v_2]] = 0$ by Jacobi identity, which implies that  $[C_L(L^2), C_L(L^2)] \subseteq Z(L)$. Since $v \in L^2$,  there exists $y \in C_L(L^2)$ such that $[x, y] = v + cu$ for some $c \in F$.
Replacing $y$ by $y-cv$, we can and do assume that $[x, y] = v$. Since ${\rm dim}[L, L] =2, {\rm dim}Z(L) =1$, we have ${\rm dim}[L/Z(L), L/Z(L)] =1$. Hence $b(L/Z(L)) = 1$ since $L$ is pure. Therefore, applying Theorem \ref{one-class} to $L/Z(L)$, we have a basis $\{x, y, v, u, w_1, \cdots , w_n\}$ for $L$ with $[x, y] = v, [x, v] = u$ and all other brackets of the basis elements in $Z(L)$. Observe that $[y, v] = [v, w_i] = 0, 1 \leq i \leq n$ since $y, w_i \in C_L(L^2)$. Also $[x, w_i ] = c_iu, c_i \in F, 1 \leq i \leq n$. Therefore, replacing $w_i$ with $w_i - c_iv$ we can assume that $[x, w_i] = 0, 1 \leq i \leq n$.

Consider the subalgebra $N = {\rm span}\{y, u, w_1, \cdots , w_n\}$ of $L$. Since $[N, N] \subseteq Z(L) = {\rm span}\{u\}$, we have $b(N) \leq 1$. If $b(N) = 0$, then $w_i \in Z(L), 1 \leq i \leq n$ since $[x, w_i] = 0 = [v, w_i]$. This is impossible since ${\rm dim}Z(L) =1$. Therefore, $b(N) = 1$ which implies $[N, N] = Z(L)$. Define the bilinear form $\varphi : N \times N \longrightarrow F$ by $\varphi(a, b) = \alpha$ if $[a, b] = \alpha u$. Then $\varphi$ is an alternating form. Note that $w_i \not\in Z(N)$ since $w_i \not\in Z(L), 1 \leq i \leq n$.
\end{proof}
The last case (see Theorem \ref{char}) is when $L$ is a finite dimensional pure nilpotent Lie algebra with $b(L) = 2$, and ${\rm dim}Z(L) =2 = {\rm dim}[L, L]$. Since $L$ is pure we have $Z(L) = [L, L]$. In this case we determine the isomorphism classes for ${\rm dim}(L) = 5, {\rm and}\, 6$ below. 
\begin{thm} Let $L$ be a $5$-dimensional nilpotent pure Lie algebra with $b(L) = 2, {\rm dim}Z(L) =2 = {\rm dim}[L, L]$. Then 
$L={\rm span}\{x_1,x_2,x_3,z_1,z_2\}$ with the nonzero products given by $[x_1, x_2] = z_1, [x_1, x_3] = z_2$.
\end{thm}
\begin{proof}
Let $x \in L$ such that $b(x) = 2$. Then there exists linearly independent vectors $y_1, y_2, z_1, z_2 \in L$ such that $[x, y_1] = z_1, [x, y_2] = z_2$. Since $[L, L] = Z(L)$, and ${\rm dim}Z(L) =2$, we have $Z(L) = {\rm span}\{z_1,z_2\}$. Suppose $[y_1, y_2] = a_1z_1 + a_2z_2$ for some $a_1, a_2 \in F$. Set $x_1 = x, x_2 =y_1 - a_2x, x_3 = y_2 + a_1x$. Then clearly we have $[x_1, x_2] = z_1, [x_1, x_3] = z_2, [x_2, x_3] = 0$ and $L={\rm span}\{x_1,x_2,x_3,z_1,z_2\}$ completing the proof.
\end{proof}
\begin{thm}\label{6-dim} Let $L$ be a $6$-dimensional nilpotent pure Lie algebra with $b(L) = 2, {\rm dim}Z(L) =2 = {\rm dim}[L, L]$. Then 
$L={\rm span}\{x_1,x_2,x_3,x_4, z_1,z_2\}$ with the nonzero products given by
\begin{enumerate}
\item $[x_1, x_2] = z_1, [x_3, x_4] = z_2$, or
\item $[x_1, x_2] = z_1, [x_2, x_3] = z_2, [x_3, x_4] = z_1, [x_1, x_4] = \alpha z_2$, for some $\alpha \in F$.
\end{enumerate} 
\end{thm}
\begin{proof}
If $L$ is a direct sum of nontrivial ideals then it follows from Proposition \ref{reducible} that  $L={\rm span}\{x_1,x_2,x_3,x_4, z_1,z_2\}$ with the nonzero products given by $(1)$. Now assume that $L$ cannot be written as direct sum of nontrivial ideals. Then we have two cases: (i) $L$ contains an element of breadth one, (ii) $L$ does not contain an element of breadth one.

Case (i): Suppose $x_1 \in L$ and $b(x_1) = 1$. Then there exist $x_2 \in L$ such that $[x_1, x_2] = z_1 \in Z(L)$ since $[L, L] = Z(L)$. Suppose $b(x_2) =1$. Then both ${\rm ker} \, {\rm ad}_{x_1}$ and ${\rm ker} \, {\rm ad}_{x_2}$ have codimension $1$ in $L$.  
Set $A = {\rm span}\{x_1, x_2\}$. Then $C_L(A) = {\rm ker} \, {\rm ad}_{x_1} \cap {\rm ker} \, {\rm ad}_{x_2}$ has codimension $2$ in $L$. So ${\rm dim} C_L(A) = 4$ and $Z(L) \subsetneq C_L(A)$. Choose $y_1, y_2 \in L$ such that $C_L(A)= Z(L) + {\rm span}\{y_1, y_2\}$. Since $b(L) =2$ and $[L, L] = Z(L)$ we have $[y_1, y_2] = z_2 \in Z(L) \setminus {\rm span}\{z_1\}$. Set $I = {\rm span}\{x_1, x_2, z_1\}$ and  $J = {\rm span}\{y_1, y_2, z_2\}$. Since $L = A \oplus C_L(A)$, we have $I , J$  ideals of $L$ and $L = I \oplus J$ which is a contradiction. Hence $b(x_2) = 2$. Thus ${\rm ker} \, {\rm ad}_{x_2}$ has codimension $2$ in $L$, so $C_L(A)$ has codimension $3$ in $L$ and $C_L(A)$ has codimension $2$ in ${\rm ker} \, {\rm ad}_{x_1}$. Choose $0 \neq x_3 \in {\rm ker} \, {\rm ad}_{x_1}$ so that ${\rm ker} \, {\rm ad}_{x_1} = C_L(A) + {\rm span}\{x_1, x_3\}$. Then $[x_1, x_3] = 0$ and $[x_2, x_3] \neq 0$. Suppose $[x_2, x_3] = \alpha z_1$ for some $0 \neq \alpha \in F$. Then $x_3 + \alpha x_1 \in C_L(A)$ which is a contradiction. Hence $[x_2, x_3] = z_2 \in Z(L)$ and $Z(L) = {\rm span}\{z_1, z_2\}$. Since $Z(L)$ has codimension $1$ in $C_L(A)$, there exists $x_4 \in C_L(A)$ such that $C_L(A) = Z(L) + {\rm span}\{x_4\}$. Thus $[x_1, x_4] = 0 = [x_2, x_4]$ and $0 \neq [x_3, x_4] \in Z(L)$. 

Case (ii): For all $x \in L \setminus Z(L)$ we have $b(x) =2$.
\end{proof}
For $\alpha \in F$, let us denote $L_{\alpha} ={\rm span}\{x_1,x_2,x_3,x_4, z_1,z_2\}$ to be the $6$-dimensional nilpotent pure Lie algebra with the nonzero products  given in  Theorem \ref{6-dim} (2). As given in (\cite{D}, page  647), for $\alpha , \beta \in F$, the Lie algebras $L_{\alpha}$ and $L_{\beta}$ are isomorphic if $\alpha = \gamma^2\beta$ for some $0 \neq \gamma \in F$. However, as the following corollary shows, this is not a necessary condition. 
\begin{cor} Let $L$ be a $6$-dimensional nilpotent pure Lie algebra over an algebraically closed field $F$ with $b(L) = 2, {\rm dim}Z(L) =2 = {\rm dim}[L, L]$. Then 
$L={\rm span}\{x_1,x_2,x_3,x_4, z_1,z_2\}$ with the nonzero products given by
\begin{enumerate}
\item $[x_1, x_2] = z_1, [x_3, x_4] = z_2$, or
\item $[x_1, x_2] = z_1, [x_2, x_3] = z_2, [x_3, x_4] = z_1$.
\end{enumerate} 
\end{cor}
\begin{proof}
It suffices to prove that for any $\alpha \in F$, $L_{\alpha} ={\rm span}\{x_1,x_2,x_3,x_4, z_1,z_2\}$ with the nonzero products  given in  
Theorem \ref{6-dim}(2) is isomorphic to $L_0$ when $F$ is algebraically closed. Since $F$ is algebraically closed $\sqrt{-\alpha} \in F$. Consider the vector $y = \sqrt{-\alpha} x_2 + x_4 \in L_{\alpha}$. Then ${\rm im}\, {\rm ad}_y = {\rm span}\{z_1 - \sqrt{-\alpha}z_2\}$. Hence $b(y) =1$. Since $L_{\alpha}$ contains an element of breadth one, as shown in the proof of Theorem \ref{6-dim} $L_{\alpha} \cong L_0$.
\end{proof}

\end{document}